\documentclass[11pt]{amsart}
\usepackage[british]{babel}

\usepackage{a4wide}
\usepackage{amssymb}
\usepackage{amsmath}
\usepackage{amsthm}
\usepackage{array}
\usepackage{enumitem}
\usepackage[hidelinks]{hyperref}
\usepackage{breakurl}
\usepackage[style=numeric,backref=true,
isbn=false,doi=true,backend=bibtex,style=alphabetic]{biblatex}
\newtheorem{thm}{Theorem}[section]

\newtheorem{cor}[thm]{Corollary}
\newtheorem{lem}[thm]{Lemma}
\newtheorem{prop}[thm]{Proposition}

\theoremstyle{definition}

\newtheorem{definition}[thm]{Definition}
\newtheorem{remark}[thm]{Remark}

\renewcommand{\epsilon}{\varepsilon}
\renewcommand{\phi}{\varphi}
\newcommand{\defeq}{\mathrel{\mathop:}=}

\DeclareMathOperator{\spt}{spt}
\DeclareMathOperator{\inte}{int}
\DeclareMathOperator{\cl}{cl}
\DeclareMathOperator{\diam}{diam}

\begin{document}

\setlist{noitemsep}


\author{Friedrich Martin Schneider}
\title[Topological entropy]{Topological entropy of continuous actions of compactly generated groups}
\date{\today}

\begin{abstract} 
  We introduce a notion of topological entropy for continuous actions of compactly generated topological groups on compact Hausdorff spaces. It is shown that any continuous action of a compactly generated topological group on a compact Hausdorff space with vanishing topological entropy is amenable. Given an arbitrary compactly generated locally compact Hausdorff topological group $G$, we consider the canonical action of $G$ on the closed unit ball of $L^{1}(G)' \cong L^{\infty}(G)$ endowed with the corresponding weak-$^{\ast}$ topology. We prove that this action has vanishing topological entropy if and only if $G$ is compact. Furthermore, we show that the considered action has infinite topological entropy if $G$ is almost connected and non-compact.
\end{abstract}


\maketitle


\section{Introduction}\label{section:introduction}

Entropy -- in its various instances -- ranges among the best-recognized and most powerful concepts within the theory of dynamical systems. Since its introduction by Adler et al. \cite{entropy}, topological entropy has earned a great deal of attention in topological dynamics. The literature provides several essentially different notions of topological entropy for continuous group actions. Among others, there is an approach towards continuous actions of finitely generated discrete groups on compact metric spaces (see for instance \cite{Ghys,Bis1,BisUrbanski,Bis2,BisWalczak,Bis3,Walczak,specification}). For a more elaborate exposition of the general topic, we refer to \cite{Downarowicz}.

In the present article we introduce and investigate a notion of topological entropy for continuous actions of compactly generated topological groups on compact Hausdorff spaces, which generalizes the above-mentioned concept for continuous actions of finitely generated discrete groups on compact metric spaces. It is shown that any continuous action of a compactly generated topological group on a compact Hausdorff space with vanishing topological entropy is amenable (Theorem~\ref{theorem:vanishing.entropy.implies.amenability}). Furthermore, given an arbitrary compactly generated locally compact Hausdorff topological group $G$, we consider the natural continuous action $\alpha$ of $G$ on the closed unit ball of $L^{1}(G)' \cong L^{\infty}(G)$ endowed with the corresponding weak-$^{\ast}$ topology. We prove that $\alpha$ has vanishing topological entropy if and only if $G$ is compact (Theorem~\ref{theorem:positive.entropy} and Corollary~\ref{corollary:positive.entropy}). Moreover, we show that $\alpha$ has infinite topological entropy if $G$ is almost connected and non-compact (Theorem~\ref{theorem:infinite.entropy} and Corollary~\ref{corollary:infinite.entropy}).

The paper is organized as follows. In Section~\ref{section:preliminaries} we recall some basic concepts and results from topological group theory and topological dynamics. In Section~\ref{section:topological.entropy} we introduce the notion of topological entropy for continuous actions of compactly generated topological groups on compact Hausdorff spaces. In Section~\ref{section:vanishing.entropy.implies.amenability} we prove that vanishing topological entropy implies amenability. Finally, in Section~\ref{section:canonical.action} we determine compute this quantity for some canonical actions of an arbitrary compactly generated topological group.


\section{Preliminaries}\label{section:preliminaries}

In this section we want to recollect some basic notation and terminology. Throughout this paper, we denote by $\mathbb{N}$ the set of natural numbers excluding $0$. As we shall constantly be concerned with continuous actions of compactly generated topological groups, we want to address some very few notational issues and recall several well-known facts and concepts from topological group theory.

Let $G$ be a topological group, i.e., $G$ is supposed to be a group equipped with a topology such that the map $G \times G \to G$, $(x,y) \mapsto x^{-1}y$ is continuous. We shall occasionally refer to the continuous maps given by $\lambda_{G}(g) \colon G \to G$, $x \mapsto gx$ where $g \in G$. Furthermore, we denote by $e_{G}$ the neutral element of $G$. By an \emph{identity neighborhood of $G$}, we mean a neighborhood of $e_{G}$ in $G$. If $H$ is a subgroup of $G$, then we turn $G/H \defeq \{ gH \mid g \in G \}$ into a topological space by endowing it with the quotient topology generated by the map $\pi_{H} \colon G \to G/H$, $g \mapsto gH$. We denote by $N(G)$ the \emph{identity component of $G$}, i.e., the connected component of $e_{G}$ in $G$. Note that $N(G)$ constitutes a closed normal subgroup of $G$. We call $G$ \emph{almost connected} if $G/N(G)$ is compact. A subset $S \subseteq G$ is said to \emph{generate $G$} if $S$ contains the neutral element of $G$ and $G = \bigcup_{n \in \mathbb{N}} S^{n}$. We call $G$ \emph{compactly generated} if there exists a compact subset $S \subseteq G$ generating $G$.

\begin{prop}[see e.g.~\cite{LieTheory}] Every almost connected, locally compact Hausdorff topological group is compactly generated. \end{prop}

As usual, if $X$ is a topological space and $S \subseteq X$, then we denote by $\inte (S)$ the \emph{interior} and by $\cl (S)$ the \emph{closure of $S$} in $X$. For later use, let us furthermore note the subsequent two basic lemmata concerning locally compact Hausdorff topological groups.

\begin{lem}\label{lemma:compact_generating_systems} Let $G$ be a locally compact Hausdorff topological group and suppose that $S \subseteq G$ is compact and generates $G$. If $T \subseteq G$ is compact, then there exists $n \in \mathbb{N}$ such that $T \subseteq \inte (S^{n})$. \end{lem}

\begin{proof} By assumption, $G = \bigcup_{n \in \mathbb{N}} S^{n}$. Furthermore, $S^{n}$ is compact and thus closed in $G$ for each $n \in \mathbb{N}$. Since $G$ is a locally compact Hausdorff space and therefore a Baire space, there exist some $m \in \mathbb{N}$ and some non-empty open subset $U \subseteq G$ such that $U \subseteq S^{m}$. As $T$ is compact and $T \subseteq GU = \bigcup_{n \in \mathbb{N}} S^{n}U$, there exists some $n \in \mathbb{N}$ such that $T \subseteq S^{n}U$, which implies that $T \subseteq \inte (S^{mn})$. \end{proof}

\begin{lem}\label{lemma:generating.subgroups.of.almost.connected.groups} Let $G$ be an almost connected, locally compact Hausdorff topological group. If $S$ is an identity neighborhood of $G$ and $\bigcup_{n \in \mathbb{N}} S^{n}$ is compact, then $G$ is compact. \end{lem}

\begin{proof} Let $S$ be an identity neighborhood of $G$ such that $K \defeq \bigcup_{n \in \mathbb{N}} S^{n}$ is compact. Clearly, there exists an open identity neighborhood $U$ of $G$ such that $U \subseteq S$ and $U^{-1} = U$. It is easy to see that $H \defeq \bigcup_{n \in \mathbb{N}} U^{n}$ is an open subgroup of $G$. Hence, $H$ is clopen in $G$. Therefore, $N(G) \subseteq H \subseteq K$. As $N(G)$ is closed in $G$, we conclude that $N(G)$ is compact. Since $G/N(G)$ is compact as well, this implies $G$ to be compact. \end{proof}

As we shall particularly be concerned with function spaces, let us furthermore set up some related notation. To begin with, we want to recall the very basics concerning means on function spaces. For a more elaborate study, we refer to \cite{AnalysisOnSemigroups}. Let $X$ be a set. We denote by $B(X)$ the set of all bounded real-valued functions on $X$. For $f \in B(X)$, we define $\Vert f \Vert_{\infty} \defeq \sup \{ \vert f(x)\vert \mid x \in X \}$. Let $H$ be a linear subspace of $B(X)$. A \emph{mean} on $H$ is a linear map $m \colon H \to \mathbb{R}$ such that \begin{displaymath}
	\inf \{ f(x) \mid x \in X \} \leq m(f) \leq \sup \{ f(x) \mid x \in X \}
\end{displaymath} for all $f \in H$. Note that, for each $x \in X$, the map $H \to \mathbb{R}, \, f \mapsto f(x)$ constitutes a mean on $H$. The set of all means on $H$ shall be denoted by $M(H)$. We endow $M(H)$ with the corresponding weak-$^{\ast}$ topology, i.e., the initial topology with respect to the maps $M(H) \to \mathbb{R}, \, m \mapsto m(f)$ where $f \in H$. Due to the Banach-Alaoglu theorem, $M(H)$ is a compact Hausdorff space.

Spaces of continuous functions will be of particular interest for our considerations. Suppose $X$ to be a topological space. We denote by $C(X)$ the set of all continuous real-valued functions on $X$. Besides, we put $C_{b}(X) \defeq C(X) \cap B(X)$. If $f \in C(X)$, then we call \begin{displaymath}
	\spt (f) \defeq \cl (\{ x \in X \mid f(x) \ne 0 \})
\end{displaymath} the \emph{support of $f$}. Furthermore, let $C_{c}(X) \defeq \{ f \in C(X) \mid \spt (f) \textnormal{ compact in } X \}$.

Finally, let us provide a brief introduction of the concept of amenability for continuous group actions. (For a more elaborate exposition, the reader is referred to \cite{brownc,runde,paterson}.) Suppose $X$ to be a compact Hausdorff space and $G$ to be a topological group. Let $\alpha$ be a \emph{continuous action of $G$ on $X$}, i.e., a continuous map $\alpha \colon G \times X \to X$ where $\alpha (e_{G},x) = x$ and $\alpha (gh,x) = \alpha (g,\alpha (h,x))$ for all $x \in X$ and $g,h \in G$. For every $g \in G$, let $\alpha_{g} \colon X \to X, \, x \mapsto \alpha (g,x)$. An \emph{$\alpha$-invariant mean} is a mean $m \colon C(X) \to \mathbb{R}$ such that $m(f) = m(f \circ \alpha_{g})$ for all $f \in C(X)$ and $g \in G$. We say that $\alpha$ is \emph{amenable} if there exists an $\alpha$-invariant mean.


\section{Topological entropy of continuous group actions}\label{section:topological.entropy}

In this section we introduce a notion of topological entropy for continuous actions of compactly generated groups and briefly discuss some of its basic properties.

For a start, let us set up some additional notation. For this purpose, let $X$ be a set. We denote by $\mathcal{P}(X)$ the set of all subsets of $X$. A subset $\mathcal{U} \subseteq \mathcal{P}(X)$ is said to be a \emph{covering} of $X$ if $X = \bigcup \mathcal{U}$. Given $\mathcal{U}, \mathcal{V} \subseteq \mathcal{P}(X)$, we say that $\mathcal{V}$ \emph{refines} $\mathcal{U}$ and write $\mathcal{U} \preceq \mathcal{V}$ if \begin{displaymath}
	\forall V \in \mathcal{V} \, \exists U \in \mathcal{U} \colon \, V \subseteq U .
\end{displaymath} Suppose $X$ to be a topological space. An \emph{open covering} of $X$ is a covering of $X$ consisting entirely of open subsets of $X$. We denote by $\mathcal{C}(X)$ the set of all finite open coverings of $X$.

\begin{definition} Let $G$ be a topological group and let $\alpha$ be a continuous action of $G$ on some compact Hausdorff space $X$. Let $\mathcal{U}, \mathcal{V} \in \mathcal{C}(X)$ and $S \subseteq G$. We say that \emph{$\mathcal{V}$ $S$-refines $\mathcal{U}$ with respect to $\alpha$} and write $\mathcal{U} \preceq_{S}^{\alpha} \mathcal{V}$ if $\alpha (g,\mathcal{V}) \defeq \{ \alpha(g,V) \mid V \in \mathcal{V} \}$ refines $\mathcal{U}$ for every $g \in S$. \end{definition}

The proof of the following lemma is elementary. However, Lemma~\ref{lemma:precompactness} justifies Definition~\ref{definition:complexity}.

\begin{lem}\label{lemma:precompactness} Let $G$ be a topological group and let $\alpha$ be a continuous action of $G$ on some compact Hausdorff space $X$. If $S \subseteq G$ is compact, then \begin{displaymath}
	\forall \mathcal{U} \in \mathcal{C}(X) \, \exists \mathcal{V} \in \mathcal{C}(X) \colon \, \mathcal{U} \preceq_{S}^{\alpha} \mathcal{V} .
\end{displaymath} \end{lem}

\begin{proof} Let $\mathcal{U} \in \mathcal{C}(X)$. We show that $\mathcal{N} \defeq \{ N \mid N \subseteq X \textnormal{ open}, \, \mathcal{U} \preceq_{S}^{\alpha} \{ N \} \}$ is a covering of $X$. To this end, let $x_{0} \in X$. Since $\alpha$ is continuous and $X$ is compact, there exist an open identity neighborhood $W$ in $G$ and a finite open covering $\mathcal{V}$ of $X$ such that \begin{displaymath}
	\forall V \in \mathcal{V} \, \exists U \in \mathcal{U} \colon \, \alpha (W,V) \subseteq U .
\end{displaymath} As $S$ is compact, there exists some finite subset $F \subseteq G$ such that $S \subseteq WF$. Due to finiteness of $F$, there is an open neighborhood $N$ of $x_{0}$ in $X$ such that $\mathcal{V} \preceq_{F}^{\alpha} \{ N \}$. Let $s \in S$. By assumption, we find some $g \in F$ where $s \in Wg$. There is $V \in \mathcal{V}$ such that $\alpha(g,N) \subseteq V$. Furthermore, there exists $U \in \mathcal{U}$ where $\alpha(W,V) \subseteq U$. We conclude that \begin{displaymath}
	\alpha(s,N) \subseteq \alpha(Wg,N) = \alpha(W,\alpha(g,N)) \subseteq \alpha(W,V) \subseteq U .
\end{displaymath} Hence, $x_{0} \in N \in \mathcal{N}$. Therefore, $\mathcal{N}$ is a covering of $X$. Since $X$ is compact, there exists a finite subset $\mathcal{V} \subseteq \mathcal{N}$ such that $X = \bigcup \mathcal{V}$. Evidently, $\mathcal{U} \preceq_{S}^{\alpha} \mathcal{V}$. \end{proof}

\begin{definition}\label{definition:complexity} Let $G$ be a topological group and let $\alpha$ be a continuous action of $G$ on some compact Hausdorff space $X$. If $S \subseteq G$ is compact and $\mathcal{U} \in \mathcal{C}(X)$, then we put \begin{displaymath}
	(S:\mathcal{U})_{\alpha} \defeq \inf \{ \vert \mathcal{V} \vert \mid \mathcal{V} \in \mathcal{C}(X), \, \mathcal{U} \preceq_{S}^{\alpha} \mathcal{V} \} .
\end{displaymath} \end{definition}

Throughout the rest of this article, we will be concerned with the exponential growth rate of the quantity introduced in Definition~\ref{definition:complexity}. More precisely, we are going to discuss the following invariant.

\begin{definition}\label{definition:topological.entropy} Let $G$ be a compactly generated topological group and assume $S$ to be a compact generating subset of $G$. Let $\alpha$ be a continuous action of $G$ on some compact Hausdorff space $X$. If $\mathcal{U}$ is a finite open covering of $X$, then we define \begin{displaymath}
	\eta(\alpha ,S, \mathcal{U}) \defeq \limsup_{n \to \infty} \frac{\log_{2} (S^{n}:\mathcal{U})_{\alpha}}{n} .
\end{displaymath} Furthermore, the \emph{topological entropy of $\alpha$ with respect to $S$} is defined to be the quantity \begin{displaymath}
	\eta(\alpha ,S) \defeq \sup \{ \eta(\alpha, S,\mathcal{U}) \mid \mathcal{U} \in \mathcal{C}(X) \} .
\end{displaymath} \end{definition}

In the case of continuous actions of finitely generated discrete groups on compact metric spaces, it is not difficult to see that our notion of topological entropy coincides with the concept considered in \cite{Ghys,Bis1,BisUrbanski,Bis2,BisWalczak,Bis3,Walczak,specification}. As one might expect, continuous actions of compact groups have vanishing topological entropy with respect to any compact generating subset.

\begin{prop}\label{proposition:vanishing.entropy} Let $G$ be a topological group and let $\alpha$ be a continuous action of $G$ on some compact Hausdorff space $X$. If $G$ is compact, then $\eta(\alpha,S) = 0$ for any compact generating subset $S \subseteq G$. \end{prop}

\begin{proof} Let $S$ be a compact generating subset of $G$. If $\mathcal{U} \in \mathcal{C}(X)$, then $(S^{n}:\mathcal{U})_{\alpha} \leq (G:\mathcal{U})_{\alpha}$ for all $n \in \mathbb{N}$, and hence \begin{displaymath}
	\eta(\alpha ,S,\mathcal{U}) = \limsup_{n \to \infty} \frac{\log_{2} (S^{n}:\mathcal{U})_{\alpha}}{n} \leq \limsup_{n \to \infty} \frac{\log_{2} (G:\mathcal{U})_{\alpha}}{n} = 0 .
\end{displaymath} This shows that $\eta(\alpha, S) = 0$. \end{proof}

Of course, the precise value of the quantity introduced in Definition~\ref{definition:topological.entropy} depends upon the choice of a compact generating system. However, we observe the following fact about continuous actions of locally compact Hausdorff topological groups, which justifies Definition~\ref{definition:entropy.type}.

\begin{prop}\label{proposition:comparing.entropy} Let $G$ be a locally compact Hausdorff topological group and let $\alpha$ be a continuous action of $G$ on some compact Hausdorff space $X$. Suppose $S,T \subseteq G$ to be compact subsets generating $G$. Then \begin{displaymath}
	\frac{1}{m} \eta(\alpha,T) \leq \eta(\alpha,S) \leq n \eta(\alpha,T) ,
\end{displaymath} where $m \defeq \inf \{ k \in \mathbb{N} \mid T \subseteq S^{k} \}$ and $n \defeq \inf \{ k \in \mathbb{N} \mid S \subseteq T^{k} \}$ (cf.~Lemma~\ref{lemma:compact_generating_systems}). \end{prop}

\begin{proof} Let $\mathcal{U} \in \mathcal{C}(X)$. Evidently, $(S^{k}:\mathcal{U}) \leq (T^{kn}:\mathcal{U})$ for all $k \in \mathbb{N}$, whence \begin{align*}
	\eta(\alpha ,S,\mathcal{U}) &= \limsup_{k \to \infty} \frac{\log_{2} (S^{k}:\mathcal{U})_{\alpha}}{k} \leq \limsup_{k \to \infty} \frac{\log_{2} (T^{kn}:\mathcal{U})_{\alpha}}{k} \\
	&= n\limsup_{k \to \infty} \frac{\log_{2} (T^{kn}:\mathcal{U})_{\alpha}}{kn} \leq n\limsup_{k \to \infty} \frac{\log_{2} (T^{k}:\mathcal{U})_{\alpha}}{k} = n \eta(\alpha ,T,\mathcal{U}) .
\end{align*} Thus, $\eta(\alpha ,S,\mathcal{U}) \leq n\eta(\alpha ,T,\mathcal{U})$. This shows that $\eta(\alpha, S) \leq n\eta(\alpha, T)$. Due to symmetry, it follows that $\eta(\alpha, T) \leq m\eta(\alpha, S)$ as well. \end{proof}

\begin{definition}\label{definition:entropy.type} Let $G$ be a compactly generated locally compact Hausdorff topological group and let $\alpha$ be a continuous action of $G$ on some compact Hausdorff space $X$. Then $\alpha$ is said to have \emph{vanishing topological entropy} if $\eta(\alpha,S) = 0$ for some (and hence any) compact generating subset $S \subseteq G$, \emph{positive topological entropy} if $\eta(\alpha,S) > 0$ for some (and hence any) compact generating subset $S \subseteq G$, \emph{finite topological entropy} if $\eta(\alpha,S) < \infty$ for some (and hence any) compact generating subset $S \subseteq G$, and \emph{infinite topological entropy} if $\eta(\alpha,S) = \infty$ for some (and hence any) compact generating subset $S \subseteq G$. \end{definition}


\section{Vanishing topological entropy implies amenability}\label{section:vanishing.entropy.implies.amenability}

In this section we prove that any continuous action of a compactly generated group on a compact Hausdorff space with vanishing topological entropy is amenable (Theorem~\ref{theorem:vanishing.entropy.implies.amenability}). For this purpose, we will utilize the following well-known combinatorial result, which is known as Hall's marriage theorem.

\begin{thm}[\cite{Hall35}]\label{theorem:hall} Let $X$ and $Y$ be finite sets and let $R \subseteq X \times Y$. Then the following are equivalent. \begin{enumerate}
	\item[(1)]	There exists an injective map $\phi \colon X \to Y$ such that $(x,\phi (x)) \in R$ for all $x \in X$.
	\item[(2)]	$\vert Z \vert \leq \vert \{ y \in Y \mid \exists x \in X \colon \, (x,y) \in R \} \vert$ for every subset $Z \subseteq X$.
\end{enumerate} \end{thm}

\begin{lem}\label{lemma:refinement.map} Let $X$ be a set. Let $\mathcal{U}$, $\mathcal{V}$, and $\mathcal{W}$ be coverings of $X$. Suppose that \begin{enumerate}
	\item[(1)]	$\mathcal{U} \cup \mathcal{V} \subseteq \mathcal{W}$,
	\item[(2)]	$\mathcal{U}$ and $\mathcal{V}$ are finite, and
	\item[(3)]	$\inf \{ \vert \mathcal{Z} \vert \mid \mathcal{Z} \subseteq \mathcal{W}, \, X = \bigcup \mathcal{W} \} = \vert \mathcal{U} \vert$.
\end{enumerate} Then there exists an injective map $\phi \colon \mathcal{U} \to \mathcal{V}$ such that $U \cap \phi (U) \ne \emptyset$ for all $U \in \mathcal{U}$. \end{lem}

\begin{proof} We shall apply Theorem~\ref{theorem:hall} to $R \defeq \{ (U,V) \in \mathcal{U} \times \mathcal{V} \mid U \cap V \ne \emptyset \}$. To this end, let $\mathcal{U}_{0} \subseteq \mathcal{U}$ and $\mathcal{V}_{0} \defeq \{ V \in \mathcal{V} \mid \exists U \in \mathcal{U}_{0} \colon \, U \cap V \ne \emptyset \}$. Evidently, $\mathcal{Z} \defeq (\mathcal{U}\setminus \mathcal{U}_{0})\cup \mathcal{V}_{0} \subseteq \mathcal{W}$. We shall prove that $\bigcup \mathcal{Z} = X$. For this purpose, let $x \in X$. If $x \in \bigcup (\mathcal{U} \setminus \mathcal{U}_{0})$, then clearly $x \in \bigcup \mathcal{Z}$. Otherwise, as $\bigcup \mathcal{U} = \bigcup \mathcal{V} = X$, there exist $U \in \mathcal{U}_{0}$ and $V \in \mathcal{V}$ such that $x \in U \cap V$, and we conclude that $V \in \mathcal{V}_{0} \subseteq \mathcal{Z}$ and therefore $x \in \bigcup \mathcal{Z}$. Now, since $\mathcal{Z}$ is a covering of $X$, it follows that $\vert \mathcal{U}\vert \leq \vert \mathcal{Z}\vert$ and hence $\vert \mathcal{U}_{0}\vert \leq \vert \mathcal{V}_{0} \vert$. Consequently, Theorem~\ref{theorem:hall} asserts that there exists an injective map $\phi \colon \mathcal{U} \to \mathcal{V}$ such that $U \cap \phi(U) \ne \emptyset$ for all $U \in \mathcal{U}$. \end{proof}

\begin{lem}\label{lemma:convergence} Let $(a_{n})_{n \in \mathbb{N}}$ be an increasing sequence in $[1,\infty)$. If $\limsup_{n \to \infty} \frac{\log_{2}(a_{n})}{n} = 0$, then $\liminf_{n \to \infty} \frac{a_{n+1}}{a_{n}} = 1$. \end{lem}

\begin{proof} The proof proceeds by contraposition. Assume that $\liminf_{n \to \infty} \frac{a_{n+1}}{a_{n}} > 1$. Then there exist $a \in (1,\infty)$ and $n_{0} \in \mathbb{N}$ such that $\frac{a_{n+1}}{a_{n}} \geq a$ for all $n \in \mathbb{N}$ with $n \geq n_{0}$. We conclude that $a_{m+n_{0}} \geq a^{m}a_{n_{0}}$ and hence \begin{displaymath}
	\frac{\log_{2} (a_{m+n_{0}})}{m+n_{0}} \geq \frac{m\log_{2} (a) + \log_{2} (a_{n_{0}})}{m+n_{0}} \geq \frac{m}{m+n_{0}}\log_{2} (a)
\end{displaymath} for all $m \in \mathbb{N}$. Since $a > 1$, it follows that $\limsup_{n \to \infty} \frac{\log_{2} (a_{n})}{n} > 0$, as desired. \end{proof}

\begin{thm}\label{theorem:vanishing.entropy.implies.amenability} Let $\alpha$ be a continuous action of a topological group $G$ on some compact Hausdorff space $X$. If $S$ is a compact generating set for $G$ and $\eta(\alpha,S) = 0$, then $\alpha$ is amenable. \end{thm}

\begin{proof} Let $\epsilon \in (0,\infty)$ and let $H \subseteq C(X)$ be finite. We observe that \begin{displaymath}
	A(H,\epsilon) \defeq \{ m \in M(C(X)) \mid \forall f \in H \, \forall s \in S \colon \vert m(f) - m(f \circ \alpha_{s^{-1}}) \vert \leq \epsilon \}
\end{displaymath} is closed in the compact Hausdorff space $M(C(X))$. We shall prove that $A(H,\epsilon) \ne \emptyset$. To this end, we put $\theta \defeq \epsilon /(1 + 2\sup_{f \in H} \Vert f \Vert_{\infty})$. Since $X$ is compact, there exists $\mathcal{U}_{0} \in \mathcal{C}(X)$ such that $\diam f(U) \leq \theta/3$ for all $U \in \mathcal{U}_{0}$ and $f \in H$. According to Lemma~\ref{lemma:precompactness}, there exists $\mathcal{U} \in \mathcal{C}(X)$ such that $\mathcal{U}_{0} \preceq_{S^{-1}}^{\alpha} \mathcal{U}$. Since $\eta (\alpha,S) = 0$, Lemma~\ref{lemma:convergence} asserts the existence of some $n \in \mathbb{N}$ such that $(S^{n+1}:\mathcal{U})_{\alpha} - (S^{n}:\mathcal{U})_{\alpha} \leq \theta (S^{n}:\mathcal{U})_{\alpha}$. Let $\mathcal{V} \in \mathcal{C}(X)$ such that $\mathcal{U} \preceq_{S^{n}}^{\alpha} \mathcal{V}$ and $\vert \mathcal{V}\vert = (S^{n} : \mathcal{U})_{\alpha}$. We observe that $\bigcup \mathcal{V}\setminus \{ V\} \ne X$ for every $V \in \mathcal{V}$. Hence, there exists a map $\pi \colon \mathcal{V} \to X$ such that $\pi(V) \in V \setminus (\bigcup \mathcal{V}\setminus \{ V \})$ for all $V \in \mathcal{V}$. Note that $\pi$ is injective. Let $F \defeq \pi (\mathcal{V})$. Of course, $m \colon C(X) \to \mathbb{R}, \, f \mapsto \frac{1}{|F|} \sum_{x \in F} f(x)$ is an element of $M(C(X))$. We are going to show that $m$ is a member of $A(H,\epsilon)$. For this purpose, let $\mathcal{W} \in \mathcal{C}(X)$ such that $\mathcal{U} \preceq_{S^{n+1}}^{\alpha} \mathcal{W}$ and $|\mathcal{W}| = (S^{n+1}:\mathcal{U})_{\alpha}$. By Lemma~\ref{lemma:precompactness}, it is true that $\mathcal{N} \defeq \{ N \mid N \subseteq X \textnormal{ open}, \, \mathcal{U} \preceq_{S^{n}}^{\alpha} \{ N \} \}$ is a covering of $X$. Note that $\mathcal{V} \cup \mathcal{W} \subseteq \mathcal{N}$ and $\inf \{ \vert \mathcal{Z} \vert \mid \mathcal{Z} \subseteq \mathcal{N}, \, X = \bigcup \mathcal{N} \} = (S^{n}: \mathcal{U})_{\alpha} = \vert \mathcal{V} \vert$. Consequently, Lemma~\ref{lemma:refinement.map} asserts the existence of an injective map $\phi \colon \mathcal{V} \to \mathcal{W}$ such that $V \cap \phi (V) \ne \emptyset$ for all $V \in \mathcal{V}$. Let $s \in S$. Since $\mathcal{U} \preceq_{S^{n}}^{\alpha} \alpha(s,\mathcal{W})$, we deduce that $\alpha(s,\mathcal{W}) \subseteq \mathcal{N}$. Therefore, Lemma~\ref{lemma:refinement.map} implies the existence of an injective map $\psi \colon \mathcal{V} \to \alpha(s,\mathcal{W})$ such that $V \cap \psi (V) \ne \emptyset$ for all $V \in \mathcal{V}$. Define $E \defeq \pi (\phi^{-1}(\alpha(s^{-1},\psi (\mathcal{V}))))$ and $\rho \colon E \to \alpha (s^{-1},F), \, x \mapsto \alpha(s^{-1},\pi(\psi^{-1}(\alpha(s,\phi(\pi^{-1}(x))))))$. Evidently, $\rho$ is injective. We argue that $|f(x)-f(\rho(x))| \leq \theta$ whenever $x \in E$ and $f \in H$. To this end, let $x \in E$. Then $x \in \pi^{-1}(x)$ by choice of $\pi$, and $\pi^{-1}(x) \cap \phi(\pi^{-1}(x)) \ne \emptyset$ by hypothesis on $\phi$. Since $x \in E$, we furthermore conclude that $\alpha(s,\phi(\pi^{-1}(x))) \in \psi(\mathcal{V})$. Due to our assumption about $\psi$, this readily implies that $\alpha(s,\phi(\pi^{-1}(x))) \cap \psi^{-1}(\alpha(s,\phi (\pi^{-1}(x)))) \ne \emptyset$, which means that $\phi(\pi^{-1}(x)) \cap \alpha(s^{-1},\psi^{-1}(\alpha(s,\phi (\pi^{-1}(x))))) \ne \emptyset$. Finally, we observe that $\rho (x) \in \alpha(s^{-1},\psi^{-1}(\alpha(s,\phi (\pi^{-1}(x)))))$. Consequently, \begin{align*}
	|f(x)&-f(\rho (x))| \leq \diam f(\pi^{-1}(x) \cup \phi (\pi^{-1}(x)) \cup \alpha (s^{-1},\psi^{-1}(\alpha (s,\phi (\pi^{-1}(x)))))) \\
	&\leq \diam f(\pi^{-1}(x)) + \diam f(\phi(\pi^{-1}(x))) + \diam f(\alpha(s^{-1},\psi^{-1}(\alpha(s,\phi (\pi^{-1}(x)))))) \leq \theta
\end{align*} for all $f \in H$. Furthermore, $|E| \geq |\mathcal{W}| - 2(|\mathcal{W}| - |\mathcal{V}|) = 2(S^{n}:\mathcal{U})_{\alpha} - (S^{n+1}:\mathcal{U})_{\alpha}$ and thus $|F| - |E| \leq (S^{n+1}:\mathcal{U})_{\alpha} - (S^{n}:\mathcal{U})_{\alpha} \leq \theta (S^{n}:\mathcal{U})_{\alpha} = \theta |F|$. Hence, if $f \in H$, then \begin{align*}
	|m(f) - m(f \circ \alpha_{s^{-1}})| &= \frac{1}{|F|}\left|\sum_{x \in F} f(x) - \sum_{x \in F} f(\alpha (s^{-1},x))\right| \\
	&= \frac{1}{|F|}\left|\sum_{x \in E} (f(x)-f(\rho (x))) + \sum_{x \in F\setminus E} f(x) - \sum_{x \in \alpha (s^{-1},F)\setminus \rho (E)} f(x)\right| \\
	&\leq \frac{1}{|F|}\left(\sum_{x \in E} |f(x)-f(\rho (x))| + \sum_{x \in F\setminus E} |f(x)| + \sum_{x \in \alpha (s^{-1},F)\setminus \rho (E)} |f(x)|\right) \\
	&\leq \theta \frac{|E|}{|F|} + 2\frac{|F| - |E|}{|F|}\Vert f\Vert_{\infty} \\
	&\leq \theta + 2\Vert f\Vert_{\infty}\theta = (1+2\Vert f \Vert_{\infty}) \theta \leq \epsilon .
\end{align*} Therefore, $m$ is a member of $A(H,\epsilon)$, and thus $A(H,\epsilon) \ne \emptyset$. Since \begin{displaymath}
	A(H_{0} \cup H_{1}, \epsilon_{0} \wedge \epsilon_{1}) \subseteq A(H_{0},\epsilon_{0}) \cap A(H_{1},\epsilon_{1})
\end{displaymath} for all finite subsets $H_{0},H_{1} \subseteq C(X)$ and $\epsilon_{0},\epsilon_{1} \in (0,\infty)$, we conclude that \begin{displaymath}
	\mathcal{A} \defeq \{ A(H,\epsilon) \mid H \subseteq C(X) \textnormal{ finite, } \epsilon \in (0,\infty) \}
\end{displaymath} has the finite intersection property. Since $M(C(X))$ is compact, it follows that $\bigcap \mathcal{A} \ne \emptyset$. Let $m \in \bigcap \mathcal{A}$. Clearly, $m(f) = m(f \circ \alpha_{s^{-1}})$ for all $s \in S$ and $f \in C(X)$. Since $S$ generates $G$, this obviously implies $m$ to be an $\alpha$-invariant mean. Therefore, $\alpha$ is amenable. \end{proof}


\section{Topological entropy of $G$ acting on $B^{\infty}(G)$}\label{section:canonical.action}

In this section we investigate the topological entropy of the canonical continuous action of an arbitrary compactly generated locally compact Hausdorff topological group $G$ on the closed unit ball of $L^{1}(G)' \cong L^{\infty}(G)$ endowed with the corresponding weak-$^{\ast}$ topology. For this purpose, we will build upon some basic abstract harmonic analysis, which may be found in \cite{HewittRoss,Deitmar}.

Throughout this section, let $\mathbb{K} \in \{ \mathbb{R}, \, \mathbb{C} \}$, let $G$ be a locally compact Hausdorff topological group and let $\mu$ be a left Haar measure on $G$. We consider the corresponding Banach space $(L^{1}(G),\Vert \cdot \Vert_{1})$. According to the Banach-Alaoglu theorem, we obtain a compact Hausdorff space by endowing the dual closed unit ball $B^{\infty}(G) \defeq B_{L^{1}(G)'}[0,1]$ with the weak-$^{\ast}$ topology, i.e., the initial topology with respect to the maps $B^{\infty}(G) \to \mathbb{K}$, $l \mapsto l(f)$ where $f \in L^{1}(G)$.

\begin{remark} Consider the isometric linear embedding $\Phi \colon (L^{\infty}(G),\Vert \cdot \Vert_{\infty}) \to (L^{1}(G),\Vert \cdot \Vert_{1})'$ given by \begin{displaymath}
	\Phi (l)(f) \defeq \int_{G} l \cdot f \, d\mu \ \ \ \ (l \in L^{\infty}(G), \, f \in L^{1}(G)) .
\end{displaymath} If $G$ is compactly generated, then $G$ is $\sigma$-finite and hence $\Phi$ is surjective. In this case we may identify $B^{\infty}(G)$ with the closed dual unit ball of the Banach space $(L^{\infty}(G),\Vert \cdot \Vert_{\infty})$. \end{remark}

In order to introduce a continuous action of $G$ on $B^{\infty}(G)$ (see Proposition~\ref{proposition:canonical.action}), we shall first recall the following basic lemma.

\begin{lem}[\cite{Deitmar}]\label{lemma:uniform.continuity} Let $G$ be a topological group. If $f \in L^{1}(G)$, $g_{0} \in G$ and $\epsilon \in (0,\infty)$, then there exists an open neighborhood $U$ of $g_{0}$ in $G$ such that $\Vert (f \circ \lambda_{G}(g_{0})) - (f \circ \lambda_{G}(g)) \Vert_{1} < \epsilon$ for all $g \in U$. \end{lem}

\begin{prop}\label{proposition:canonical.action} The map $\alpha \colon G \times B^{\infty}(G) \to B^{\infty}(G)$ given by \begin{displaymath}
	\alpha (g,l)(f) \defeq l(f \circ \lambda_{G}(g) ) \ \ \ \ (g \in G, \, l \in B^{\infty}(G), \, f \in L^{1}(G)) 
\end{displaymath} constitutes a continuous action of $G$ on $B^{\infty}(G)$. \end{prop}

\begin{proof} First we substantiate that $\alpha$ is well defined. To this end, let $l \in B^{\infty}(G)$. Of course, $\alpha(g,l)$ is linear, and \begin{displaymath}
	\vert \alpha (g,l)(f) \vert = \vert l(f \circ \lambda_{G}(g)) \vert \leq \Vert l \Vert \cdot \Vert f \circ \lambda_{G}(g) \Vert_{1} \leq \Vert f \circ \lambda_{G}(g) \Vert_{1} = \Vert f \Vert_{1}
\end{displaymath} for all $f \in L^{1}(G)$. Thus, $\alpha(g,l) \in B^{\infty}(G)$. To show continuity, let $g_{0} \in G$, $l_{0} \in B^{\infty}(G)$, $f \in L^{1}(G)$ and $\epsilon \in (0,\infty)$. Now, $V \defeq \{ l \in B^{\infty}(G) \mid \vert (l_{0}-l)(f \circ \lambda_{G}(g_{0})) \vert < \frac{\epsilon}{2} \}$ is an open neighborhood of $l_{0}$ in $B^{\infty}(G)$. Besides, by Lemma~\ref{lemma:uniform.continuity}, there exists an open neighborhood $U$ of $g_{0}$ in $G$ such that $\Vert (f \circ \lambda_{G}(g_{0})) - (f \circ \lambda_{G}(g)) \Vert_{1} < \frac{\epsilon}{2}$ for every $g \in U$. Therefore, for all $g \in U$ and $l \in V$, we conclude that \begin{align*}
	\vert (\alpha (g_{0},l_{0}) - \alpha (g,l))(f) \vert &\leq \vert (\alpha (g_{0},l_{0}) - \alpha (g_{0},l))(f)\vert + \vert (\alpha (g_{0},l) - \alpha (g,l))(f) \vert \\
	&= \vert (l_{0}-l)(f \circ \lambda_{G}(g_{0}))\vert + \vert l((f \circ \lambda_{G}(g_{0})) - (f \circ \lambda_{G}(g)))\vert \\
	&< \frac{\epsilon}{2} + \Vert l \Vert \cdot \Vert (f \circ \lambda_{G}(g_{0})) - (f \circ \lambda_{G}(g)) \Vert_{1} \\
	&< \frac{\epsilon}{2} + \frac{\epsilon}{2} = \epsilon .
\end{align*} This shows that $\alpha$ is continuous. Of course, $\alpha (e_{G},l) = l$ and $\alpha (g_{0}g_{1},l) = \alpha (g_{0},\alpha (g_{1},l))$ for all $g_{0},g_{1} \in G$ and $l \in B^{\infty}(G)$. Hence, $\alpha$ constitutes a continuous action of $G$ on $B^{\infty}(G)$. \end{proof}

We aim at determining the entropy type of $\alpha$ with regard to Definition~\ref{definition:entropy.type}. As it turns out, $\alpha$ has positive topological entropy unless $G$ is compact (see Corollary~\ref{corollary:positive.entropy}). In order to prove this, we need the following auxiliary result. 

\begin{lem}\label{lemma:uniformly_discrete_product_sequences} Suppose $S$ to be a compact identity neighborhood of $G$ such that $\bigcup_{n \in \mathbb{N}} S^{n}$ is not compact. Then there is a sequence $(s_{n})_{n \in \mathbb{N}} \in S^{\mathbb{N}}$ such that $s_{m}\cdots s_{n} \notin \inte(S^{m-n})$ for all $m,n \in \mathbb{N}$ with $m > n$. \end{lem}

\begin{proof} For every $k \in \mathbb{N}$, let us define \begin{displaymath}
	A_{k} \defeq \{ s \in S^{\mathbb{N}} \mid \forall m,n \in \{ 1,\ldots ,k \} \colon \, m>n \Longrightarrow s_{m}\cdots s_{n} \notin \inte (S^{m-n}) \} .
\end{displaymath} Evidently, $(A_{k})_{k \in \mathbb{N}}$ is a descending chain of closed subsets of $S^{\mathbb{N}}$. Furthermore, we observe that this sequence does not contain the empty set: Consider any arbitrary natural number $k \in \mathbb{N}$. Since $H$ is not compact, it follows that $S^{k+1} \nsubseteq S^{k}$. Let $s_{1},\ldots,s_{k+1} \in S$ such that $s_{k+1}\cdots s_{1} \notin S^{k}$, and define $s_{l} \defeq e_{G}$ for all $l \in \mathbb{N}$ with $l > k+1$. We argue that $(s_{l})_{l \in \mathbb{N}} \in A_{k}$. Otherwise there were $m,n \in \{ 1,\ldots,k \}$ such that $m>n$ and $s_{m}\cdots s_{n} \in \inte(S^{m-n}) \subseteq S^{m-n}$, wherefore $s_{k+1}\cdots s_{1} = s_{k+1}\cdots s_{m+1}s_{m}\cdots s_{n}s_{n-1}\cdots s_{1} \in S^{k-m+1}S^{m-n}S^{n-1} = S^{k}$, which would yield a contradiction. Hence, as $S^{\mathbb{N}}$ is compact, we conclude that $\bigcap_{k \in \mathbb{N}} A_{k} \ne \emptyset$. This proves the claim. \end{proof}

\begin{lem}\label{lemma:entropy.for.neighborhoods} Suppose $S$ to be a compact identity neighborhood of $G$ such that $\bigcup_{n \in \mathbb{N}} S^{n}$ is not compact. Then there is a two-element open covering $\mathcal{V}$ of $B^{\infty}(G)$ such that $(S^{2n}:\mathcal{V})_{\alpha} \geq 2^{n}$ for all $n \in \mathbb{N}$. \end{lem}

\begin{proof} Let $U$ be an open identity neighborhood in $G$ such that $UU \subseteq S$ and $U^{-1} = U$. By Lemma~\ref{lemma:uniformly_discrete_product_sequences}, there exists a sequence $(s_{n})_{n \in \mathbb{N}} \in S^{\mathbb{N}}$ such that $s_{m}\cdots s_{n} \notin \inte (S^{m-n})$ for all $m,n \in \mathbb{N}$ with $m > n$. Define $t_{n} \defeq s_{2n}s_{2n-1}$ for all $n \in \mathbb{N}$. According to Urysohn's lemma, there exists $f \in C_{c}(G)$ such that $\Vert f \Vert_{\infty} = 1$, $f(e_{G}) = 1$ and $K \defeq \spt (f) \subseteq U$. Of course, $c \defeq \int_{G} f^{2} \, d\mu > 0$. Consider the two-element open covering $\mathcal{V} \defeq \{ V_{0}, \, V_{1} \}$ of $B^{\infty}(G)$ where $V_{0} \defeq \{ l \in B^{\infty}(G) \mid \vert l(f) \vert < \frac{2c}{3} \}$ and $V_{1} \defeq \{ l \in B^{\infty}(G) \mid \vert l(f) \vert > \frac{c}{3} \}$. Let $n \in \mathbb{N}$. We show that $(S^{2n}:\mathcal{V})_{\alpha} \geq 2^{n}$. To this end, consider any $\delta \in \{ 0,1 \}^{n}$. Define $f_{\delta} \in C_{c}(G)$ by \begin{displaymath}
	f_{\delta} \defeq \sum_{j = 1}^{n} \delta (j) \cdot (f \circ \lambda_{G}(t_{j}\cdots t_{1})) .
\end{displaymath} We observe that $\spt (f \circ \lambda_{G}(t_{i}\cdots t_{1})) = (t_{i}\cdots t_{1})^{-1}(\spt f) = t_{1}^{-1}\cdots t_{i}^{-1}K$ for every $i \in \mathbb{N}$. Let $i,j \in \mathbb{N}$ such that $i < j$. Then \begin{align*}
	t_{i+1}^{-1}\cdots t_{j}^{-1} &= (s_{2i+2}s_{2i+1})^{-1}\cdots (s_{2j}s_{2j-1})^{-1} = s_{2i+1}^{-1}\cdots s_{2j}^{-1} \\
	&= (s_{2j}\cdots s_{2i+1})^{-1} \notin (\inte (S^{2(j-i)-1}))^{-1} .
\end{align*} Besides, $KK^{-1} \subseteq UU^{-1} = (UU)^{-1} \subseteq (\inte (S))^{-1} \subseteq (\inte (S^{2(j-i)-1}))^{-1}$. Hence, we conclude that $t_{i+1}^{-1}\cdots t_{j}^{-1} \notin KK^{-1}$ and therefore \begin{align*}
	\spt (f \circ \lambda_{G}(t_{i}\cdots t_{1})) \cap \spt (f \circ \lambda_{G}(t_{j}\cdots t_{1})) &= (t_{1}^{-1}\cdots t_{i}^{-1}K) \cap (t_{1}^{-1}\cdots t_{j}^{-1}K) \\
	&= t_{1}^{-1}\cdots t_{i}^{-1}(K \cap (t_{i+1}^{-1}\cdots t_{j}^{-1}K)) = \emptyset .
\end{align*} Thus, $(f \circ \lambda_{G}(t_{i}\cdots t_{1})) \cdot (f \circ \lambda_{G}(t_{j}\cdots t_{1})) = 0$. This particularly implies that $\Vert f_{\delta} \Vert_{\infty} \leq 1$. Accordingly, \begin{displaymath}
	l_{\delta} \colon L^{1}(G) \to \mathbb{K}, \, h \mapsto \int_{G} f_{\delta} \cdot h \, d\mu
\end{displaymath} is a member of $B^{\infty}(G)$. We compute that \begin{align*}
	\alpha (t_{i}\cdots t_{1},l_{\delta})(f) &= \int_{G} f_{\delta} \cdot (f \circ \lambda_{G}(t_{i}\cdots t_{1})) \, d\mu \\
	&= \sum_{j = 1}^{n} \delta (j) \cdot \int_{G} (f \circ \lambda_{G}(t_{i}\cdots t_{1})) \cdot (f \circ \lambda_{G}(t_{j}\cdots t_{1})) \, d\mu \\
	&= \delta (i) \cdot \int_{G} (f \circ \lambda_{G}(t_{i}\cdots t_{1})) \cdot (f \circ \lambda_{G}(t_{i}\cdots t_{1})) \, d\mu \\
	&= \delta (i) \cdot \int_{G} f^{2} \, d\mu = \delta (i) \cdot c
\end{align*} and furthermore conclude that \begin{displaymath}
	\alpha (t_{i}\cdots t_{1},l_{\delta}) \in V_{0} \ \Longleftrightarrow \ \delta (i) = 0 \ \Longleftrightarrow \ \delta (i) \neq 1 \ \Longleftrightarrow \ \alpha (t_{i}\cdots t_{1},l_{\delta}) \notin V_{1} 
\end{displaymath} for every $i \in \{ 1,\ldots,n \}$. Now, let $\mathcal{W}$ be any finite open covering of $B^{\infty}(G)$ such that $\mathcal{W}$ $S^{2n}$-refines $\mathcal{V}$ with respect to $\alpha$. We shall argue that $|\mathcal{W}| \geq 2^{n}$. For each $\delta \in \{ 0,1 \}^{n}$, fix some $Z_{\delta} \in \mathcal{W}$ where $l_{\delta} \in Z_{\delta}$. Consider distinct elements $\delta_{0},\delta_{1} \in \{ 0,1 \}^{n}$. There exists $i \in \{ 1,\ldots,n \}$ where $\delta_{0}(i) \ne \delta_{1}(i)$. Without loss of generality, suppose that $\delta_{0}(i) = 0$. As $l_{\delta_{0}} \in Z_{\delta_{0}}$ and $\alpha(t_{i}\cdots t_{1},l_{\delta_{0}}) \notin V_{1}$, it follows that $\alpha(t_{i}\cdots t_{1},Z_{\delta_{0}}) \nsubseteq V_{1}$ and thus necessarily $\alpha(t_{i}\cdots t_{1},Z_{\delta_{0}}) \subseteq V_{0}$. However, $\alpha(t_{i}\cdots t_{1},Z_{\delta_{1}}) \nsubseteq V_{0}$ as $l_{\delta_{1}} \in Z_{\delta_{1}}$ and $\alpha(t_{i}\cdots t_{1},l_{\delta_{1}}) \notin V_{0}$. Therefore, $Z_{\delta_{0}} \ne Z_{\delta_{1}}$. This establishes that $|\mathcal{V}| \geq 2^{n}$. Accordingly, $(S^{2n}:\mathcal{V})_{\alpha} \geq 2^{n}$. \end{proof}

\begin{thm}\label{theorem:positive.entropy} Suppose that $G$ is compactly generated. Let $S$ be a compact generating subset of $G$ and let $k \defeq \inf \{ n \in \mathbb{N} \mid S^{n} \textit{ identity neighborhood of } G \}$ (cf.~Lemma~\ref{lemma:compact_generating_systems}). If $G$ is not compact, then there exists a two-element open covering $\mathcal{V}$ of $B^{\infty}(G)$ such that \begin{displaymath}
	\eta (\alpha,S,\mathcal{V}) \geq \frac{1}{2k} .
\end{displaymath} \end{thm}

\begin{proof} Let $T \defeq S^{k}$. By Lemma~\ref{lemma:entropy.for.neighborhoods}, there exists a two-element open covering $\mathcal{V}$ of $B^{\infty}(G)$ such that $(T^{2n}:\mathcal{V})_{\alpha} \geq 2^{n}$. For all $n \in \mathbb{N}$, it follows that $(S^{2kn}:\mathcal{V})_{\alpha} = (T^{2n}:\mathcal{V})_{\alpha} \geq 2^{n}$. Hence, \begin{displaymath}
	\eta(\alpha ,S) \geq \eta (\alpha , S, \mathcal{V}) \geq \limsup_{n \to \infty} \frac{\log_{2} (S^{2kn}:\mathcal{V})_{\alpha }}{2kn} \geq \frac{1}{2k} .\qedhere
\end{displaymath} \end{proof}

\begin{cor}\label{corollary:positive.entropy} Suppose $G$ to be compactly generated. Then $G$ is compact if and only if $\alpha$ has vanishing topological entropy. \end{cor}

\begin{proof} If $G$ is compact, then $\alpha$ has vanishing topological entropy due to Proposition~\ref{proposition:vanishing.entropy}. Conversely, if $\alpha$ has vanishing topological entropy, then $G$ is compact by Theorem~\ref{theorem:positive.entropy}. \end{proof}

As we are going to see, $\alpha$ has infinite topological entropy provided that $G$ is almost connected and non-compact (see Corollary~\ref{corollary:infinite.entropy}).

\begin{thm}\label{theorem:infinite.entropy} Suppose $G$ to be almost connected. Let $S$ be a compact generating subset of $G$. If $G$ is not compact, then \begin{displaymath}
	\sup \{ \eta(\alpha, S,\mathcal{V}) \mid \mathcal{V} \in \mathcal{C}(B^{\infty}(G)), \, |\mathcal{V}| = 2 \} = \infty .
\end{displaymath} \end{thm} 

\begin{proof} Let $m \in \mathbb{N}$. By Lemma~\ref{lemma:compact_generating_systems}, there exists $k \in \mathbb{N}$ such that $T \defeq S^{k}$ is an identity neighborhood in $G$. Furthermore, there exists a compact identity neighborhood $U$ of $G$ such that $U^{2km} \subseteq T$. Due to Lemma~\ref{lemma:generating.subgroups.of.almost.connected.groups}, it follows that $\bigcup_{n \in \mathbb{N}} U^{n}$ is not compact. Hence, by Lemma~\ref{lemma:entropy.for.neighborhoods}, there exists a finite open covering $\mathcal{V}$ of $B^{\infty}(G)$ such that $(U^{2n}:\mathcal{V})_{\alpha} \geq 2^{n}$. For all $n \in \mathbb{N}$, it follows that $(S^{2kn}:\mathcal{V})_{\alpha} = (T^{2n}:\mathcal{V})_{\alpha} \geq (U^{4kmn}:\mathcal{V})_{\alpha} \geq 2^{2kmn}$. Hence, \begin{displaymath}
	\eta(\alpha, S, \mathcal{V}) \geq \limsup_{n \to \infty} \frac{\log_{2} (S^{2kn}:\mathcal{V})_{\alpha}}{2kn} \geq m .\qedhere
\end{displaymath} \end{proof}

\begin{cor}\label{corollary:infinite.entropy} Suppose $G$ to be almost connected. Then $G$ is non-compact if and only if $\alpha$ has infinite topological entropy. \end{cor}

\begin{proof} This is an immediate consequence of Proposition~\ref{proposition:vanishing.entropy} and Theorem~\ref{theorem:infinite.entropy}. \end{proof}


\small

\sloppy

\printbibliography

\end{document}